\documentclass{article}

\usepackage{amsmath}
\usepackage{mathrsfs}
\usepackage{amssymb}
\usepackage{amsthm}

\usepackage{authblk}

\usepackage{geometry}
\usepackage{color}
\usepackage[normalem]{ulem}
\newtheorem{lemma}{Lemma}[section]
\newtheorem{theorem}[lemma]{Theorem}
\newtheorem{corollary}[lemma]{Corollary}
\newtheorem{proposition}[lemma]{Proposition}
\newtheorem{definition}[lemma]{Definition}

\def\C{{\mathbb C}}
\def\Z{{\mathbb Z}}
\def\1{{\bf 1}}
\def \w{\omega}
\def\CC{{\mathcal{C}}}
\def \Or{{\bf O}}
\def \Hom{{\rm Hom}}
\def \KW {{\rm KW}}
\def\DD{{\mathcal{D}}}
\def\id{{\rm id}}

\begin{document}
	\title{A unitary vertex operator algebra arising from the $3C$-algebra}
	
	\author{Xiangyu Jiao}
	\affil{School of Mathematical Sciences,  East China Normal University, Shanghai 200241, China \\
		Key Laboratory of MEA (Ministry of Education),  East China Normal University, Shanghai 200241, China \\
		Shanghai Key Laboratory of PMMP,  East China Normal University, Shanghai 200241, China}
	\author{Wen Zheng
		\footnote{Corresponding author\\
			Email address: $wzheng14@qdu.edu.cn$\\
			Xiangyu Jiao is supported by the  Science and Technology Commission of Shanghai Municipality (Grant No. 23ZR1418600, No. 22DZ2229014).
			Wen Zheng is supported by the NSFC No.~12201334 and the National Science Foundation of Shandong Province No.~ZR2022QA023.}}
	\affil{School of Mathematics and Statistics,
		Qingdao University, 266071, China\\
		Qingdao International Mathematics Center, Qingdao, 266071, China}
	
	\maketitle
	
	\begin{abstract}

We give an algebraic proof of the unitarity of the vertex operator algebra $L(21/22, 0)\oplus L(21/22, 8)$ and of all its irreducible ordinary modules, using a coset realization arising from the $3C$-algebra. Motivated by the structure of the resulting module decomposition, we establish a general result on fusion rules for commutant vertex operator subalgebras within the framework of modular tensor categories. As an application of this general result, we explicitly determine the  fusion rules of all irreducible $L(21/22, 0)\oplus L(21/22, 8)$-modules.

	\end{abstract}
	
	\setcounter{tocdepth}{2}
	\tableofcontents

	\newpage
	
	\section{Introduction}
	The definition of a unitary vertex operator algebra was first introduced in \cite{DLin1}. Roughly speaking, a vertex operator algebra (VOA for short) is called unitary if it is equipped with a positive definite Hermitian form that is compatible with the vertex operator structure. There are two main motivations for introducing this notion. One arises from its close relationship with unitary representations of infinite dimensional Lie algebras, while the second arises from its fundamental role in relating the algebraic formulation of two-dimensional conformal field theory via vertex operator algebras to the analytic approach based on conformal nets. We refer the reader to \cite{DLin1} for a detailed discussion. For other approaches to unitary VOAs, see B. Gui's works \cite{Gui1, Gui2, Gui3, Gui4} and the references therein.
	
	It was shown in \cite{DLin1} that many well-known rational VOAs, such as the unitary series of Virasoro VOAs and lattice VOAs associated to positive definite even lattices, are unitary VOAs. In \cite{DLin2}, preunitary VOAs with central charge $c<1$ were classified, yielding a one to one correspondence with the classification of conformal nets with $c<1$ \cite{KL}. In the same work, it was proved there that all such preunitary vertex operator algebras are unitary except for two cases $L(c_9, 0)\oplus L(c_9, h_{1,7}^9)$ and $L(c_{27}, 0)\oplus L(c_{27},h_{1,11}^{27})\oplus L(c_{27},h_{1,19}^{27})\oplus L(c_{27},h_{1,29}^{27})$,
	 whose unitarity remained open from a purely vertex operator algebraic perspective. To bridge these algebraic and analytic classifications, Gui \cite{Gui3} established a fundamental relation between unitary VOA extensions  and Q-systems \cite{Longo} in $C^*$-tensor categories. This framework allowed him to prove that the two exceptional VOAs above are indeed unitary, thereby  completing the classification of unitary VOAs with $c<1$.  See also \cite{CGGH}.
 Moreover, Gui's work also shows that all such VOAs are completely unitary \cite{Gui3}, which in particular implies that all their irreducible ordinary modules are unitary modules.
	
	 The main purpose of the present paper is twofold. First, we give an independent and purely algebraic proof of the unitarity of the vertex operator algebra $L(21/22, 0)\oplus L(21/22, 8)$ and of all its irreducible ordinary modules. Our approach is based on a concrete coset realization arising from the $3C$-algebra \cite{LYY}. Second, motivated by a structural feature observed in the explicit decomposition of the $3C$-algebra, we establish a general result on fusion rules for coset vertex operator subalgebras within the framework of modular tensor categories, building in particular on the methods developed in
\cite{DRX2}.

	The paper is organized as follows: In Section 2, we briefly recall the construction of the $3C$-algebra and its irreducible modules following \cite{LYY}.
	In Section 3, we review the notions of unitary VOAs and  unitary modules introduced in \cite{DLin1}, and then we use the concrete construction of $3C$-algebra to show that the VOA $L(21/22, 0)\oplus L(21/22, 8)$ and all  its irreducible modules are unitary.
	Section 4 is devoted to fusion rules: after recalling necessary notions such as the Kac-Wakimoto set and results from \cite{DRX2}, we establish our result on fusion rules for commutant subalgebras.  As a corollary, we derive the fusion rules of irreducible $L(21/22, 0)\oplus L(21/22, 8)$-modules.
	
	We assume that the reader is familiar with the basic theory of vertex operator algebras (cf. \cite{FHL}, \cite{FLM}, \cite{LL}) and modular tensor categories (cf. \cite{EGNO, KO}).
	
	\section{Constructions of the $3C$-algebra and its irreducible modules}
	In this section, we briefly recall several standard notions of vertex operator algebras used throughout the paper, and then review the construction of the 3C-algebra and its irreducible modules following \cite{LYY}.
	
	A vertex operator algebra $V$ is said to be \emph{rational} if every admissible $V$-module is completely reducible. Rationality plays a fundamental role in the study of fusion rules and modular invariance in vertex operator algebra theory. Moreover, $V$ is called $C_2$-cofinite if the quotient space $V/{C_2(V)}$ where $$C_2(V)=\text{span}\{u_{-2}v\mid u,v\in V\}$$ is finite dimensional.
	Rationality and the $C_2$-cofiniteness condition ensure good finiteness properties of modules and are closely related to the modularity of characters. They also lead to many important results in VOA theory.

	Let $(V, Y, \1, \w)$ be a vertex operator algebra and
	$(U, Y, \1, \w')$ is a vertex operator subalgebra of $V$. Set $\w^{''}=\w-\w'$, $Y(\w',z)=\sum_{i\in \Z}L'(i)z^{-i-2} $ and $Y(\w'',z)=\sum_{i\in \Z}L''(i)z^{-i-2}$. Note that $L'(0)\mid_U=L(0)\mid_U$. The \emph{coset (or commutant)} $U^c$ of $U$ is defined to be
	$$U^c=\{u\in V\mid v_n u=0, v\in U,n\geq 0\}$$
	(see \cite{FZ}). $U^c$ can be viewed as the space of vacuum-like vectors for $U$ \cite{L1}:
	$$U^c=\{u\in V\mid L'(-1)u=0\}.$$
	It is well known that $(U^c, Y, \1, \w-\w')$ is a vertex operator subalgebra of
	$V$ \cite{FZ,LL}.

	Let $L(c_m, 0)$ denote the simple vertex operator algebra associated to the Virasoro algebra with central charge
	\begin{equation}\label{eq:VIRC}
		c=c_m=1-\frac{6}{(m+2)(m+3)}, \qquad m=1,2,3, \dots.
	\end{equation}
	The irreducible modules of the VOA $L(c_m,0)$ are given by the Virasoro minimal models $L(c_m,h^m_{r,s})$, where
	\begin{equation}\label{eq:VIRH}
		h^m_{r,s} = \frac{(r(m+3)-s(m+2))^2 - 1}{4(m+2)(m+3)},\qquad 1 \le
		r \le m+1,\quad 1 \le s \le m+2.
	\end{equation}
	Note that $h^m_{r,s} = h^m_{m+2-r,m+3-s}$ and that
	$L(c_m,h^m_{r,s})$, $1 \le s \le r \le m+1$ are all
	inequivalent irreducible $L(c_m,0)$-modules.
	Let  $V_{\sqrt{2}E_8}$ be the lattice VOA associated to the lattice $\sqrt{2}E_8$ and $\alpha_1,\alpha_2,\dots,\alpha_8$ be the simple roots of $E_8$. Following the convention in \cite{LYY}, let $-\alpha_0$ be the highest root, then we have:
	\begin{equation}\label{eq:A_0}
		\alpha_0+2\alpha_1+3\alpha_2+4\alpha_3+ 5\alpha_4+6\alpha_5+4\alpha_6+2\alpha_7+3\alpha_8 =
		0.
	\end{equation}
	Let $L$ be the sublattice of $E_8$ root lattice generated by
	$\alpha_j, 0 \le j \le 7$. Then $L\cong A_8$. Let $\Phi$ be the root system of $L$, $h$ the Coxeter number of $\Phi$, $\omega$ the Virasoro element of $V_{\sqrt{2}L}$, which is also the Virasoro element of $V_{\sqrt{2}E_8}$. We define two vectors in $V_{\sqrt{2}E_8}$ as follows:
	\begin{equation}\label{sw}
		\begin{split}
			s=s(\Phi)&=\frac{1}{2(h+2)}\sum_{\alpha\in \Phi^+}\left(
			\alpha(-1)^2\cdot
			1 -2(e^{\sqrt{2}\alpha}+ e^{-\sqrt{2}\alpha})\right),\\
			\tilde{\omega}=\tilde{\omega}(\Phi)&= \omega - s.
		\end{split}
	\end{equation}
	It is shown in \cite{DLMN} that $\tilde{\omega}$ and $s$ are mutually
	orthogonal conformal vectors and the central charge of
	$\tilde{\omega}$ is $16/11$.  Define $U$ to be the coset (or commutant) VOA associated with the conformal vector $s$.
	\begin{equation}\label{3c}
		U=\{v\in V_{\sqrt{2}E_8}\,| \, s_1 v=0 \}.
	\end{equation}
	This VOA is known as the $3C$-algebra introduced in \cite{LYY},  which is related to the $3C$ conjugacy class of the product of two $2A$-involutions of the Monster group \cite{C}  and can be decomposed into irreducible $L(\frac{1}2,0)\otimes L(\frac{21}{22},0)$-module:
	\[
	\begin{split}
		U_{3C}
		& \cong  L(\frac{1}2,0)\otimes L(\frac{21}{22},0)
		\oplus L(\frac{1}{2},0)\otimes L(\frac{21}{22},8)
		\\
		& \quad \oplus  L(\frac{1}2,\frac{1}2)\otimes L(\frac{21}{22},\frac{7}{2})
		\oplus L(\frac{1}2,\frac{1}2)\otimes L(\frac{21}{22},\frac{45}{2})
		\\
		& \quad  \oplus L(\frac{1}{2},\frac{1}{16})\otimes L(\frac{21}{22},\frac{31}{16})
		\oplus L(\frac{1}{2},\frac{1}{16})\otimes L(\frac{21}{22},\frac{175}{16}).
	\end{split}
	\]
	
	For simplicity, we use $[c, h]$ to denote $L(c, h)$ throughout the paper. Then we have the following result (Theorem 3.38 in \cite{LYY}).
	\begin{theorem}\label{3cmod}
		There are exactly five irreducible $U$-modules $U(2k), 0 \le k \le
		4$. In fact, $U(0)=U$ and as $L(1/2, 0)\otimes L(21/22,0)$-modules,
		\begin{align*}
			U(2) & \cong [0,\frac{13}{11}] \oplus [0,\frac{35}{11}] \oplus
			[\frac{1}{2},\frac{15}{22}] \oplus [\frac{1}{2},\frac{301}{22}]
			\oplus [\frac{1}{16},\frac{21}{176}] \oplus [\frac{1}{16},
			\frac{901}{176}],\\
			U(4) & \cong [0,\frac{6}{11}]  \oplus [0,\frac{50}{11}] \oplus
			[\frac{1}{2},\frac{1}{22}] \oplus [\frac{1}{2},\frac{155}{22}]
			\oplus [\frac{1}{16},\frac{85}{176}] \oplus [\frac{1}{16},
			\frac{261}{176}],\\
			U(6) & \cong  [0,\frac{1}{11}] \oplus [0,\frac{111}{11}] \oplus
			[\frac{1}{2},\frac{35}{22}] \oplus [\frac{1}{2},\frac{57}{22}]
			\oplus [\frac{1}{16}, \frac{5}{176}]
			\oplus [\frac{1}{16},\frac{533}{176}],\\
			U(8) &\cong [0,\frac{20}{11}] \oplus [0,\frac{196}{11}] \oplus
			[\frac{1}{2},\frac{7}{22}] \oplus [\frac{1}{2},\frac{117}{22}]
			\oplus [\frac{1}{16},\frac{133}{176}] \oplus
			[\frac{1}{16},\frac{1365}{176}].
		\end{align*}
	\end{theorem}
	
	Note that the lattice VOA $V_{\sqrt{2}E_8}$ has a decomposition as follows (equation (3.95) in \cite{LYY}):
	
	\begin{align}\label{E8decom}
		V_{\sqrt{2}E_8} \cong
		& \bigoplus_{ { 0\leq k_j\leq j+1}
			\atop{{ k_j \equiv 0\,\mathrm{mod}\, 2} \atop{ j=0,1,\dots, 8}}}
		L(c_1,h^1_{k_{0}+1,k_1+1})\otimes \cdots  L(c_8,h^8_{k_{7}+1,k_8+1})
		\otimes U(k_8),
	\end{align}
	where $c_m$ and $h_{r,s}^m$ are given by \eqref{eq:VIRC} and
	\eqref{eq:VIRH}, and $U(k_8)$ is the corresponding $U_{3C}$-module as in Theorem \ref{3cmod}. This decomposition will be crucial in Section 3 to deduce the unitarity of irreducible modules.

\section{Unitarity of the VOA $L(21/22, 0)\oplus L(21/22, 8)$ and its modules}

In this section, first we shall briefly recall the definitions of unitary VOA and unitary modules from \cite{DLin1}. Then we show that the VOA $L(21/22, 0)\oplus L(21/22, 8)$ is unitary and all its irreducible modules are unitary modules.
In the following we only consider the vertex operator algebra $(V, Y, \1, \omega)$ of \emph{CFT-type}, i.e. $V_n=0, n<0$ and $V_0=\mathbb{C}\1$.
\begin{definition}
	Let $(V, Y, \1, \omega)$ be a vertex operator algebra. An anti-linear automorphism $\phi$ of $V$ is
	an anti-linear isomorphism (as anti-linear map) $\phi:V\to V$ such
	that $\phi(\bf1)=\bf1, \phi(\omega)=\omega$ and
	$\phi(u_nv)=\phi(u)_n\phi(v)$ for any $ u, v\in V$ and $n\in
	\mathbb{Z}$.
\end{definition}

\begin{definition}
	Let $(V, Y, \1, \omega)$ be a vertex operator algebra and $\phi: V\to V$ be an anti-linear involution, i.e. an anti-linear automorphism of order $2$. The pair $(V, \phi)$ is called  a unitary vertex operator algebra if there exists a positive definite Hermitian form $(,): V\times V\to \mathbb{C}$
	which is $\C$-linear on the first vector and anti-$\C$-linear on the second vector such that the following invariant property holds:\\
	$$(Y(e^{zL(1)}(-z^{-2})^{L(0)}a, z^{-1})u, v)=(u, Y(\phi(a), z)v)$$
	for any $a, u, v\in V$, where $L(n)$ is defined by $Y(\omega, z)=\sum_{n\in \Z}L(n)z^{-n-2}$.
\end{definition}

Next we give the definition of a unitary module as in \cite{DLin1}.
\begin{definition}\label{unimod} Let $(V, Y, \1, \omega)$ be a vertex operator algebra and $\phi$ an anti-linear involution of $V$. An ordinary $V$-module $(M, Y_M)$ is called a unitary $V$-module if there exists a positive definite
	Hermitian form $(,)_M: M\times M\to \mathbb{C}$ which is $\C$-linear on the first vector and anti-$\C$-linear on the second vector such that the following invariant property holds:	$$(Y_M(e^{zL(1)}(-z^{-2})^{L(0)}a, z^{-1})w_1, w_2)_M=(w_1, Y_M(\phi(a),z)w_2)_M$$  for $a \in V $ and $w_1, w_2\in M$.
\end{definition}

We shall frequently use the following result (Corollary 2.8 in \cite{DLin1}):
\begin{proposition}\label{coset}
	Let $(V,\phi)$ be a unitary VOA and $(U, Y, \1, \w')$ be a vertex operator subalgebra of $V$ such that  $\phi(\w')=\w'$. Then the coset  $(U^c, \phi|_{U^c})$ VOA is also a unitary VOA.
\end{proposition}

We now turn to prove the unitarity of $L(21/22, 0)\oplus L(21/22, 8)$.
\begin{theorem}\label{c9unitary}
	The VOA $L(21/22, 0)\oplus L(21/22, 8)$ is unitary.
\end{theorem}
\begin{proof}
	Let $L$ be a positive definite even lattice. Then the lattice VOA $V_L$ is unitary and the anti-linear involution $\phi: V_L \to V_L$ is given by:
	$$\alpha_1(-n_1)\cdots\alpha_k(-n_k)\otimes e^{\alpha}\mapsto (-1)^k\alpha_1(-n_1)\cdots\alpha_k(-n_k)\otimes
	e^{-\alpha},$$
	see \cite{DLin1}.
	In particular, the lattice VOA $V_{\sqrt{2}E_8}$ is unitary, and the conformal vector defined in equation (\ref{sw}) satisfies $\phi(s)=s$.  Then by (\ref{3c}) and Proposition \ref{coset}, the associated coset VOA $U_{3C}$ is unitary with the anti-involution $\phi|_{U_{3C}}$.
	
	From Lemma 3.32 in \cite{LYY}, the conformal vector $\omega_1$ of the subalgebra $L(1/2, 0)$ in $U_{3C}$ is fixed by $\phi$. Explicitly, we have
	\begin{equation*}
		\omega_1=\frac{11}{32}\tilde{\omega}+\frac{1}{32}\sum_{\alpha \in \alpha_8+A_8, <\alpha, \alpha>=2}(e^{\sqrt{2}\alpha}+e^{-\sqrt{2}\alpha}),
	\end{equation*}
	where $\alpha_8$ is given in (\ref{eq:A_0}) and $\tilde{\omega}$ is given in (\ref{sw}). Since $L(21/22, 0)\oplus L(21/22, 8)$ can be realized as the coset subalgebra of $U_{3C}$ associated to $\w_1$, Proposition \ref{coset} implies that it is unitary.
\end{proof}

Next we show that all irreducible  $L(21/22, 0)\oplus L(21/22, 8)$-modules are unitary. For simplicity, we write $$\mathcal{M}=L(21/22, 0)\oplus L(21/22, 8).$$

Uniqueness of the VOA $\mathcal M$ is given in Theorem 4.11 of \cite{DLin2} (see also Theorem 4.4 in \cite{DZh}). All irreducible $\mathcal{M}$-modules are given in \cite{DLin2} using the modular invariant of type $(A_{10}, E_6)$. Comparing those module structures with the decomposition of the $U_{3C}$-modules given in Theorem \ref{3cmod}, we notice that all irreducible $\mathcal{M}$-modules arise from the decomposition of the $U_{3C}$-modules. For each irreducible $U_{3C}$-module $U(2k)$, we rewrite the decomposition in Theorem \ref{3cmod} as:
\begin{equation}\label{U2k}
	U(2k)\cong L(\frac{1}{2}, 0)\otimes \mathcal{M}_{k0}\oplus L(\frac{1}{2}, \frac{1}{2})\otimes \mathcal{M}_{k1}\oplus L(\frac{1}{2}, \frac{1}{16})\otimes \mathcal{M}_{k2},
\end{equation}
where $\mathcal{M}_{kl}, k\in\{0, 1, 2, 3, 4\}, l\in\{0, 1, 2\}$ are precisely all the irreducible $\mathcal{M}$-module and $\mathcal{M}\cong \mathcal{M}_{00}$.

Before we prove our result, we need the following lemma:
\begin{lemma}\label{cosetmod}
	Let $V, U,$ and $U^c$ be VOAs satisfying the assumptions of Proposition \ref{coset}. In addition, we assume both $U$ and $U^c$ are rational. Let $(N, Y_N)$ be an irreducible unitary $V$-module. Suppose that as a $U\otimes U^c$-module, $$N\cong\oplus_{i=0}^{m}U_i\otimes M_i,$$
	where $U_i, M_i$ are irreducible $U$-modules and $U^c$-modules respectively. Then each $M_i$ is a unitary $U^c$-module.
\end{lemma}
\begin{proof}
	From Proposition \ref{coset}, we know both $(V,\phi)$ and the coset $(U^c, \phi|_{U^c})$ are unitary VOAs. Since 
	$(N, Y_N)$ is a unitary $V$-module, there exists a positive definite
	Hermitian form $(,)_N: N\times N\to \mathbb{C}$  satisfying the following invariant property:
	\begin{equation}\label{invariant}
		(Y_N(e^{zL(1)}(-z^{-2})^{L(0)}a, z^{-1})w_1, w_2)_N=(w_1, Y_N(\phi(a),z)w_2)_N
	\end{equation}
	for $a\in V, w_1, w_2\in N$.
	
	Restricting this identity to $a\in U^c$ and $w_1, w_2\in M_i$, and using the fact that $L'(1)a=L'(0)a=0$ for any $a\in U^c$, we obtain
	\begin{equation}\label{invariant2}
		(Y_{M_i}(e^{zL^{''}(1)}(-z^{-2})^{L^{''}(0)}a, z^{-1})w_1, w_2)_{M_i}=(w_1, Y_{M_i}(\phi(a),z)w_2)_{M_{i}},
	\end{equation}
	where  $Y_{M_i}$ and $(\, , \,)_{M_i}$ are the restrictions of $Y_N$ and $(\, ,\,)_{N}$ to $M_i$ respectively. Hence $M_i$ is a unitary $U^c$-module, which completes the proof.
\end{proof}

\begin{theorem}
	Any irreducible $L(21/22, 0)\oplus L(21/22, 8)$-module is unitary.
\end{theorem}
\begin{proof}
	From above discussions, it suffices to show that for any $i\in\{0, 1, 2\}, k\in\{0, 1, 2, 3, 4\}$, $\mathcal{M}_{ik}$ is a unitary module. By the proof of Theorem \ref{c9unitary} and the decomposition given in equation
	(\ref{E8decom}), we have that each irreducible $U_{3C}$-module $U(2k)$ is a unitary module, for $k\in\{0, 1, 2, 3, 4\}$. Applying  Lemma \ref{cosetmod} to the decomposition (\ref{U2k}), we conclude that each $\mathcal{M}_{i,k}$ is a unitary $\mathcal{M}$-module, and this completes the proof.
\end{proof}

\section{Fusion rules}
In this section, instead of computing the fusion rules of irreducible $L(21/22, 0)\oplus L(21/22, 8)$-modules directly, we establish a general result on fusion rules for commutant subalgebras under suitable assumptions by applying the categorical framework developed in \cite{DRX2}. The fusion rules of irreducible $L(21/22, 0)\oplus L(21/22, 8)$-modules are then an immediate consequence of this general result.

First we recall several notions and results from \cite{DRX2}.
Throughout this section, we assume that $U$ is a vertex operator subalgebra of the vertex operator algebra $V$ such that

(1) $U=U^{cc},$

(2) $U, U^c$, and $V$ are rational, $C_2$-cofinite and of CFT types, and satisfy $L(1)V_1=0$,

(3) The conformal weight of any irreducible $U$-module, $U^c$-module, or $V$-module is positive, except for the vacuum modules $U$, $U^c$, and $V.$

Let $M^i$, for $i \in I=\{1,...,p\}$, denote the inequivalent irreducible $V$-modules with $M^1=V$.
Let $W^{\alpha}$, for $\alpha\in J=\{1,...,q\}$, denote the inequivalent irreducible $U$-modules with $W^1=U,$ and let $N^{\phi}$, for $\phi\in K=\{1,...,s\}$, denote the inequivalent irreducible $U^c$-modules
with $N^1=U^c.$ Denote by $\CC_1=\CC_U$  the category of $U$-modules, by $\CC_2=\CC_{U^c}$ the category of $U^c$-modules, and by
$\CC=\CC_V=(\CC_1\otimes \CC_2)_V^0$ the category of $V$-modules.  Under the above assumptions, $\CC_1$ and $\CC_2$ are pseudo unitary modular tensor categories \cite{ENO}. We use $\Or(\CC_1)=\{W^\alpha|\alpha\in J\}$  to denote the isomorphism classes of simple objects of $\CC_1$ and $U=W^1=1_{\CC_1}.$  Then $V\in \CC_1\otimes \CC_2$ is a regular commutative algebra \cite{HKL}. Moreover, the category of local $V$-modules $\CC=(\CC_1\otimes \CC_2)_V^0$ is also a modular tensor category \cite{HKL}. For each $i\in I$, the object $M^i$ decomposes in $\CC_1\otimes \CC_2$ as follows:
\begin{equation}\label{maindecom}
	M^i\cong\bigoplus_{\alpha\in J_i}W^\alpha\otimes M^{(i,\alpha)},
\end{equation}
where $J_i$ is a subset of $J.$   For $\beta\in J$, set
$$a_{\beta\otimes 1}=V\boxtimes_{\CC_1\otimes \CC_2}(W^\beta\otimes 1_{\CC_2}).$$
and recall the Kac-Wakimoto set (see \cite{DRX2})
$$\KW=\{W^\beta\in \Or(\CC_1)|a_{\beta\otimes 1}\in\CC\}.$$
Note that $1_{\CC_1}\in \KW.$ So the Kac-Wakimoto set is not empty. Let ${\DD}=\{W^\beta|\beta\in J_1\}.$  The \emph{M\"uger centralizer} $C_{\CC_1}(\DD)$ is the  subcategory of $\CC_1$ consisting of the objects $Y$ in $\CC_1$ such that $c^1_{Y,X}\circ  c^1_{X,Y} = \id_{X\boxtimes Y}$  for all $X$ in $\DD$ where $c^1_{Y,X}:Y\boxtimes X\to X\boxtimes Y$ is the braiding isomorphism.
We shall use the following results from \cite{DRX2} (cf. Proposition 4.6 and 4.9 and Theorem 6.5 therein, adapted to our setting):
\begin{proposition}\label{muger} With the above assumptions and notations, $a_{\alpha\otimes 1}\in \CC$  if and only if $W^\alpha\in C_{\CC_1}(\DD).$
\end{proposition}

\begin{proposition}\label{kacW}  If $\KW=\{1_{\CC_1}\}$, then the objects $M^{(i,\alpha)} (i\in I, \alpha\in J)$ form a complete set of inequivalent simple objects in $\CC_2.$
\end{proposition}

\begin{theorem}\label{cvalue}
	Let $d_{\alpha}$ be the categorical dimension of $W^{\alpha}$, then
	\begin{equation*}
		\sum_{\beta\in J_1}d_{\beta}^2\leq 	\sum_{\gamma\in J_i}d_{\gamma}^2
	\end{equation*}
	for any $i\in I$.
\end{theorem}

From now on, we assume $J_1=J$. First we have the following result:
\begin{proposition}\label{simpleobjects}
	All notation and assumptions are as in the beginning of this section. Furthermore, we assume $J_1=J$. Then for any $i\in I$, we have $J_i=J$ and the set $\{M^{(i,\alpha)}| i\in I, \alpha\in J\}$ is the complete set of inequivalent simple objects in $\CC_2.$
\end{proposition}
\begin{proof}
	Since $J_1=J$, the first assertion is an immediate consequence of Theorem \ref{cvalue}. In this case, the \emph{M\"uger centralizer} $C_{\CC_1}(\DD)$ coincides with the \emph{M\"uger center} of $\CC_1$. So it consists only one element $U=W^1=1_{\CC_1}$ due to the fact that the category $\CC_1$ is a modular tensor category. Then an application of Proposition \ref{muger} and \ref{kacW} yields the result.
\end{proof}

Next we turn to compute fusion rules $N_{M^{(i,\alpha)},M^{(j,\beta)}}^{M^{(k,\gamma)}}$. Let $\tilde{\CC}$ be one of the categories $\CC, \CC_1, \CC_2$, and let $X$ be an object of $\tilde{\CC}$, we use $\dim_{\tilde{C}}X$ to denote the categorical dimension of $X$. 	Notice that under our settings, the categorical dimension equals the Frobenius-Perron dimension, which is exactly the quantum dimension studied in \cite{DJX} and \cite{DRX1}. For simplicity, in the following we will use the symbol $\dim_{\tilde{C}}X$ to denote any one of these dimensions. The following lemmas are needed for the proof of our main result:
\begin{lemma}\label{catedim}
	Under the assumptions of Proposition \ref{simpleobjects}, for any $i\in I$, the categorical dimension $\dim_{\CC} M^i$ equals the categorical dimension $\dim_{\CC_2} M^{(i,1)}$.
\end{lemma}
\begin{proof}
	First, by Lemma 4.1 in \cite{DRX2}, we have $V\boxtimes_{\CC_1\otimes \CC_2}(W^1\otimes M^{(i,1)})$ is a simple object in $(\CC_1\otimes \CC_2)_V$. Since $(\CC_1\otimes \CC_2)_V$ is a fusion category, we have
	\begin{eqnarray*}
		\Hom_V(V\boxtimes_{\CC_1\otimes \CC_2}(W^1\otimes M^{(i,1)}), M^{i})
		&\cong& \Hom_{\CC_1\otimes \CC_2}\left(W^1\otimes M^{(i,1)}, \sum_{\alpha\in J}W^{\alpha}\otimes M^{(i,\alpha)}\right)\\
		&=& \Hom_{\CC_1\otimes \CC_2}\left(W^1\otimes M^{(i,1)}, W^1\otimes M^{(i,1)}\right),
	\end{eqnarray*}
	so we get $V\boxtimes_{\CC_1\otimes \CC_2}(W^1\otimes M^{(i,1)})$ is a simple object in $(\CC_1\otimes \CC_2)_V^0$ and $M^i\cong V\boxtimes_{\CC_1\otimes \CC_2}(W^1\otimes M^{(i,1)})$ as a $V$-module.
	Thus, $$\dim_{\CC}M^i=\dim_{\CC}V\boxtimes_{\CC_1\otimes \CC_2}(W^1\otimes M^{(i,1)})=\dim_{\CC_1}W^1\cdot\dim_{\CC_2}M^{(i,1)}=\dim_{\CC_2}M^{(i,1)}.$$
\end{proof}

\begin{lemma}
	Under the assumptions of Proposition \ref{simpleobjects}, for any $i, j, k\in I$, we have
	\begin{equation*}
		N_{M^i, M^j}^{M^k}=N_{ M^{(i,1)},M^{(j,1)}}^{M^{(k,1)}}.
	\end{equation*}
\end{lemma}
\begin{proof}
	By Proposition 2.9 in \cite{ADL}, we have
	\begin{equation}\label{fusioninequal}
		N_{M^i, M^j}^{M^k}\leq N_{W^1\otimes M^{(i,1)},W^1\otimes M^{(j,1)}}^{M^k}=N_{W^1\otimes M^{(i,1)},W^1\otimes M^{(j,1)}}^{W^1\otimes M^{(k,1)}}=N_{M^{(i,1)},M^{(j,1)}}^{M^{(k,1)}}.
	\end{equation}
	Moreover, applying (\ref{fusioninequal}) and Lemma \ref{catedim} we have
	\begin{eqnarray*}
		\dim_{\CC}M^{i}\cdot \dim_{\CC}M^{j}
		&=&\dim_{\CC}(M^i\boxtimes_{\CC}M^j)\\
		&=&\dim_{\CC}\left(\sum_{k\in I}N_{M^i, M^j}^{M^k}M^k\right)\\
		&=&\sum_{k\in I}N_{M^i, M^j}^{M^k}\dim_{\CC}M^k\\
		&\leq&\sum_{k\in I}N_{M^{(i,1)},M^{(j,1)}}^{M^{(k,1)}}\dim_{\CC_2}M^{(k,1)}\\
		&=&\dim_{\CC_2}\left(\sum_{k\in I}N_{M^{(i,1)},M^{(j,1)}}^{M^{(k,1)}}M^{(k,1)}\right)\\
		&=&\dim_{\CC_2}\left(M^{(i,1)}\boxtimes_{ \CC_2}M^{(j,1)}\right)\\
		&=&\dim_{\CC_2}M^{(i,1)}\cdot \dim_{ \CC_2}M^{(j,1)}\\
		&=&\dim_{\CC}M^{i}\cdot \dim_{\CC}M^{j}.
	\end{eqnarray*}
	Hence $N_{M^i, M^j}^{M^k}=N_{ M^{(i,1)},M^{(j,1)}}^{M^{(k,1)}}.$
	
\end{proof}

\begin{lemma}
	Under the assumptions of Proposition \ref{simpleobjects} (without requiring $J_1=J$), for any $\alpha, \beta, \gamma \in J$, we have
	\begin{equation*}
		N_{W^{\alpha}, W^{\beta}}^{W^{\gamma}}=N_{ M^{(1,\alpha)},M^{(1, \beta)}}^{M^{(1, \gamma)}} \quad\text{and} \quad \dim_{\CC_1}W^{\alpha}=\dim_{\CC_2}M^{(1,\alpha)}.
	\end{equation*}
\end{lemma}
\begin{proof}
	This result follows from Theorem 4.2 in \cite{DRX2}; see also Theorem 3.1 in \cite{Lin}.
\end{proof}

\begin{lemma}
	Under the assumptions of Proposition \ref{simpleobjects}, for any $i\in I, \alpha\in J$, we have
	\begin{equation*}
		M^{(i,1)}\boxtimes_{ \CC_2}M^{(1,\alpha)}\cong M^{(i,\alpha)}.
	\end{equation*}
\end{lemma}
\begin{proof}
	First, by an argument similar to that in the proof of Lemma \ref{catedim}, we have $M^i\cong V\boxtimes_{\CC_1\otimes \CC_2}(W^1\otimes M^{(i,1)})$ as $V$-modules. Then as a $U\otimes U^c$-module, we have
	\begin{eqnarray*}
		M^{i}
		&\cong&\left(\sum_{\alpha \in J}W^{\alpha}\otimes M^{(1,\alpha)}\right)\boxtimes_{\CC_1\otimes \CC_2}(W^1\otimes M^{(i,1)})\\
		&\cong&\sum_{\alpha \in J}\left(W^{\alpha}\boxtimes_{\CC_1} W^{1}\right)\otimes\left(M^{(1,\alpha)}\boxtimes_{\CC_2} M^{(i,1)}\right)\\
		&\cong&\sum_{\alpha \in J}W^{\alpha}\otimes\left(M^{(1,\alpha)}\boxtimes_{\CC_2} M^{(i,1)}\right).
	\end{eqnarray*}
	Hence $M^{(i,1)}\boxtimes_{ \CC_2}M^{(1,\alpha)}\cong M^{(i,\alpha)}.$
	
\end{proof}
Now we prove our main result:
\begin{theorem}\label{fusionrule1}
	Under the assumptions of Proposition \ref{simpleobjects},  for any $i, j, k \in I$, and $\alpha, \beta, \gamma \in J$, we have the fusion rules $$N_{M^{(i,\alpha)},M^{(j,\beta)}}^{M^{(k,\gamma)}}=N_{M^i,M^j}^{M^k} N_{W^{\alpha},W^{\beta}}^{W^{\gamma}}.$$
\end{theorem}
\begin{proof}
	Applying the above several lemmas, we have
	\begin{eqnarray*}
		M^{(i,\alpha)}\boxtimes_{ \CC_2}M^{(j,\beta)} &\cong&(M^{(i,1)}\boxtimes_{ \CC_2}M^{(1,\alpha)})
		\boxtimes_{ \CC_2}(M^{(j,1)}\boxtimes_{ \CC_2}M^{(1,\beta)})\\
		&\cong&(M^{(i,1)}\boxtimes_{ \CC_2}M^{(j,1)})
		\boxtimes_{ \CC_2}(M^{(1,\alpha)}\boxtimes_{ \CC_2}M^{(1,\beta)})\\
		&\cong&\left(\sum_{k\in I}N_{ M^{(i,1)},M^{(j,1)}}^{M^{(k,1)}}M^{(k, 1)}\right)\boxtimes_{ \CC_2}
		\left(\sum_{\gamma\in J}N_{ M^{(1,\alpha)},M^{(1, \beta)}}^{M^{(1, \gamma)}}M^{(1,\gamma)}\right)\\
		&\cong&\left(\sum_{k\in I}N_{M^i,M^j}^{M^k}M^{(k, 1)}\right)\boxtimes_{ \CC_2}
		\left(\sum_{\gamma\in J}N_{W^{\alpha},W^{\beta}}^{W^{\gamma}}M^{(1,\gamma)}\right)\\
		& =&\sum_{\substack{k\in I \\ \gamma \in J}}N_{M^i,M^j}^{M^k} N_{W^{\alpha},W^{\beta}}^{W^{\gamma}}\left(M^{(k, 1)}\boxtimes_{ \CC_2}M^{(1,\gamma)}\right)\\
		&\cong&\sum_{\substack{k\in I \\ \gamma \in J}}N_{M^i,M^j}^{M^k} N_{W^{\alpha},W^{\beta}}^{W^{\gamma}}M^{(k,\gamma)}.
	\end{eqnarray*}
	Hence $N_{M^{(i,\alpha)},M^{(j,\beta)}}^{M^{(k,\gamma)}}=N_{M^i,M^j}^{M^k} N_{W^{\alpha},W^{\beta}}^{W^{\gamma}}.$
\end{proof}

Next we give a proof of the fusion rules of irreducible $L(21/22,0)\oplus L(21/22, 8)$-modules. For simplicity, we will use $(r, s)$ to denote irreducible $L(21/22,0)$-module $L(c_9, h_{r,s}^9)$, $1 \le s \le r \le 10$. We recall the notion of admissible triples introduced in \cite{W}. In our settings, $p$, $q$ in the following definition will be $11$, $12$ respectively.
\begin{definition}
	Assume $p, q\in \{2, 3, 4, \cdots\}$ and $p, q $ are relatively prime. An ordered triple of pairs of integers $((r,s), (r',s'), (r^{''}, s^{''}))$  is called admissible if $0<r,r',r^{''}<p$, $0<s,s',s^{''}<q,$ $r+r'+r^{''}<2p$, $s+s'+s^{''}<2q,$ $r+r'>r^{''},$ $r+r^{''}>r',$ $r'+r^{''}>r,$ $s+s'>s^{''},$ $s+s^{''}>s'$, $s'+s^{''}>s$, and both $r+r'+r^{''}$ and $s+s'+s^{''}$ are odd.
\end{definition}

The fusion rules for $U_{3C}$-modules are expressed in terms of admissible triples \cite{DZh}.
\begin{proposition}
	The fusion rules among the irreducible $U_{3C}$-modules $U(2k)$, for $k=0,\dots, 4,$ are as follows:
	$$U(2i)\boxtimes U(2j)=\oplus_{k}U(2k),$$
	where the summation over $k$ is for $i, j, k \in \{0,1,2,3,4\}$ satisfying $((2i-1,1), (2j-1,1),(2k-1,1))$ is an admissible triple.
\end{proposition}\label{3Cfusion}
As in equation (\ref{U2k}), all inequivalent irreducible $L(21/22,0)\oplus L(21/22, 8)$-modules are denoted by $\mathcal{M}_{i,l}$, where $i\in \{0,1,2,3,4\}$ and $l\in \{0,1,2\}$. Then we have the following:
\begin{corollary}  The fusion rules of $\mathcal M$-modules are:
	\begin{align*}
		&\mathcal{M}_{i,0}\boxtimes \mathcal{M}_{j,l}=\oplus_k \mathcal{M}_{k,l},\quad \quad \mathcal{M}_{i,1}\boxtimes \mathcal{M}_{j,1}=\oplus_k \mathcal{M}_{k,0},\\
		&\mathcal{M}_{i,1}\boxtimes \mathcal{M}_{j,2}=\oplus_k \mathcal{M}_{k,2},\quad \quad \mathcal{M}_{i,2}\boxtimes \mathcal{M}_{j,2}=\oplus_k \mathcal{M}_{k,0} \bigoplus \oplus_k \mathcal{M}_{k,1},
	\end{align*}
	where the summation over $k$ is for $i, j, k \in \{0,1,2,3,4\}$ satisfying $((2i-1,1), (2j-1,1),(2k-1,1))$ is an admissible triple.
\end{corollary}
\begin{proof}
	This is a direct consequence of the fusion rules of $L(1/2,0)$-modules (see \cite{DMZ,W}), the fusion rules of irreducible $U_{3C}$-modules given in the previous proposition and Theorem \ref{fusionrule1}.
\end{proof}

\section*{Acknowledgments}

We thank Tiziano Gaudio for his valuable communication and helpful comments.


\end{document}